\documentclass[13pt,a4paper,reqno]{amsart}
\usepackage{amssymb}
\usepackage{amsmath,amsthm}
\usepackage{mathrsfs}
\usepackage{bbm}
\usepackage{ulem}

\def\bc{\begin{center}}
\def\ec{\end{center}}
\def\be{\begin{equation}}
\def\ee{\end{equation}}

\def\N{\mathbb N}

\newtheorem{lem}{Lemma}[section]

\newtheorem{thm}[lem]{Theorem}

\theoremstyle{remark}

\numberwithin{equation}{section}

\begin{document}
\title[Multifractal analysis of Saint-Petersburg potential]
{Multifractal analysis of the Birkhoff sums of Saint-Petersburg potential}

\author{Dong Han KIM}
\address{Dong Han KIM\\ Department of Mathematics Education, Dongguk University-Seoul,30 Pildong-ro 1-gil, Jung-gu, Seoul, 04620 Korea}
\email{kim2010@dongguk.edu}

\author{Lingmin LIAO }
\address{Lingmin LIAO\\  LAMA, Universit\'e Paris-Est Cr\'eteil Val de Marne, 61, avenue du G\'en\'eral de Gaulle
94010 Cr\'eteil Cedex, France}
\email{lingmin.liao@u-pec.fr}

\author{Micha\l\ RAMS}
\address{Micha\l\ Rams\\Institute of Mathematics\\ Polish Academy of Sciences\\ ul.
\'Sniadeckich 8, 00-656 Warszawa\\ Poland }
\email{rams@impan.pl}

\author{Bao-Wei WANG}
\address{Bao-Wei WANG\\ School of Mathematics and Statistics, Huazhong University of Science and Technology, 430074 Wuhan, China}
\email{bwei\_wang@hust.edu.cn }



\thanks{D.K. was supported by the NRF of Korea (NRF-2015R1A2A2A01007090, NRF-2017K1A3A1A21013650).
M.R. was  supported  by  National  Science  Centre  grant 2014/13/B/ST1/01033 (Poland).
B.W. was partially supported by NSFC N0. 11471030.}

\maketitle

\begin{abstract}
Let $((0,1], T)$ be the doubling map in the unit interval and $\varphi$ be the Saint-Petersburg potential, defined by
$\varphi(x)=2^n$ if $x\in (2^{-n-1}, 2^{-n}]$ for all $n\geq 0$.
We consider asymptotic properties of the Birkhoff sum $S_n(x)=\varphi(x)+\cdots+\varphi(T^{n-1}(x))$.
With respect to the Lebesgue measure, the Saint-Petersburg potential is not integrable and it is known that $\frac{1}{n\log n}S_n(x)$ converges to $\frac{1}{\log 2}$ in probability.
We determine the Hausdorff dimension of the level set $\{x: \lim_{n\to\infty}S_n(x)/n=\alpha\} \ (\alpha>0)$, as well as that of the set $\{x: \lim_{n\to\infty}S_n(x)/\Psi(n)=\alpha\} \ (\alpha>0)$, when $\Psi(n)=n\log n$, $n^a $ or $2^{n^\gamma}$ for $a>1$, $\gamma>0$.
The fast increasing Birkhoff sum of the potential function $x\mapsto 1/x$ is also studied.
\end{abstract}

\section{Introduction}\label{sec:intro}

Let $T$ be the doubling map on the unit interval $(0,1]$ defined by
\begin{equation*}
Tx= 2x-\lceil 2x\rceil +1,
\end{equation*}
where $\lceil x \rceil$ is the smallest integer larger than or equal to $x$.
 Let $\epsilon_1$ be the function defined by $\epsilon_1(x)=\lceil 2x \rceil-1$ and $\epsilon_n(x):=\epsilon_1(T^{n-1}x)$ for $n\ge 2$.
 Then each real number $x\in (0,1]$ can be expanded into {\it an infinite series} as \begin{equation}\label{f0}
x=\frac{\epsilon_1(x)}{2}+\cdots+\frac{\epsilon_n(x)}{2^n}+\cdots.
\end{equation} We call (\ref{f0}) the binary expansion of $x$ and also write it 
as $$x=[\epsilon_1(x)\epsilon_2(x)\dots].$$
The Saint-Petersburg potential is a function $\varphi: (0,1]\to \mathbb{R}$ defined as
\[
\varphi(x)=2^n \ \text{if} \ x\in (2^{-n-1}, 2^{-n}],  \ \forall n\geq 0.\]
We remark that the definition of $\varphi$ is equivalent to
$$
\varphi(x)=2^{n}\ \ {\text{where $n\geq 0$ is the  smallest integer such that}}\ \epsilon_{n+1}(x)=1.
$$
and is also equivalent to $$
\varphi(x)=2^{n}\ \  {\text{if the binary expansion of}}\ x \ {\text{begins with}}\ 0^{n}1,
$$
where $0^{n} (n\geq 0)$ means a block with $n$ consecutive zeros.

The name of Saint-Petersburg potential is motivated by the famous Saint-Petersburg game in probability theory.
The Saint-Petersburg potential is of infinite expectation with respect to the Lebesgue measure.
Furthermore, it increases exponentially fast near to the point $0$.

In this paper, we are concerned with the following Birkhoff sums of the Saint-Petersburg potential:
$$
\forall n \geq 1, \quad S_n(x):=\varphi(x)+\varphi(T(x)) + \cdots+\varphi(T^{n-1}(x)), \qquad  x\in (0,1].
$$

Let $$
I=\{x\in (0,1]: \epsilon_1(x)=1\}.
$$
Define the hitting time of $x\in (0,1]$ to $I$ as $$
n(x):=\inf\{n\ge 0: T^{n}x\in I\}.
$$
Then
$$
n(x)=n \quad \text{if} \ x\in\left({1\over 2^{n+1}}, {1\over 2^{n}}\right], \ \text{for all } n\geq 0.
$$
Using $n(x)$, we define a new dynamical system
$\widehat{T}: (0,1]\to (0,1]$ by
$$\widehat T (x) = T^{n(x)+1}(x)=2^{n+1}\Big(x-\frac{1}{2^{n+1}}\Big) \ \text{if} \ x\in \left({1\over 2^{n+1}}, {1\over 2^{n}}\right], \ \text{for all } n\geq 0,$$
called the acceleration of $T$, in order that $\varphi$ and $\varphi \circ \widehat T$ are independent.
Let
$$
\widehat S_n(x):=\varphi(x)+\varphi(\widehat T (x))+\cdots+\varphi(\widehat T^{n-1}(x)), \qquad  x\in (0,1].
$$
The convergence in probability of $\widehat S_n(x)$ is well known (e.g.  \cite[p.253]{Fe68}) which states that
for any $\epsilon>0$, the Lebesgue measure $\lambda$ of
$$
\left\{x\in (0,1]\ :  \ \Big| \frac{\widehat S_n(x)}{n\log n}-\frac{1}{\log 2} \Big| \ge \epsilon \right\}
$$
tends to $0$ as $n\to \infty$.

Let $\{\Psi_n\}_{n\ge 1}$ be an increasing sequence such that $\Psi_n\to \infty$ as $n\to\infty$.
Then it was shown in \cite{Fe46} that Lebesgue almost surely either
$$
\lim_{n\to \infty}\frac{\widehat S_n(x)}{\Psi_n}=0 \quad {\text{or}} \quad  \limsup_{n\to \infty}\frac{\widehat S_n(x)}{\Psi_n}=\infty,
$$
according as $$\sum_{n\ge 1} {\lambda}( \{ x \in (0,1]: \varphi(x) \ge \Psi_n \} )<\infty \quad {\text{or}} \quad =\infty.$$

Let $n_1=n_1(x)=n(x)+1$ and $n_k=n_k(x)=n_1(\widehat T^{k-1}x)=n(\widehat T^{k-1}x)$ for $k\geq 2$. It is direct to see that
\[
\forall \ell\geq 1, \quad S_{n_1+\cdots+ n_{\ell}}(x)=2\widehat S_{\ell}(x)-\ell.
\]
Moreover, the ergodicity of $T$ (of $\widehat T$) implies Lebesgue almost surely
\[
\lim_{\ell\to\infty}{n_1(x)+\cdots +n_{\ell}(x) \over \ell} =\int_0^1 (n(x)+1)d\lambda(x)=2.
\]
Combining these two facts together, we obtain the same convergence results as above if we replace $\widehat S_n$ by $S_n$.
In particular, the average $S_n(x)/(n\log n)$ converges to $1/\log 2$ in probability, and
 almost surely (with respect to the Lebesgue measure) either
$$
\lim_{n\to \infty}\frac{ S_n(x)}{\Psi_n}=0 \quad {\text{or}} \quad  \limsup_{n\to \infty}\frac{ S_n(x)}{\Psi_n}=\infty,
$$
according as $$\sum_{n\ge 1} {\lambda}( \{ x \in (0,1]: \varphi(x) \ge \Psi_n \} )<\infty \quad {\text{or}} \quad =\infty,$$
where $\{\Psi_n\}_{n\ge 1}$ is an increasing sequence such that $\Psi_n\to \infty$ as $n\to\infty$. Recall that $\varphi$ has infinite expectation with respect to the Lebesgue measure.
We thus have $S_n(x)/n$ converges to infinity for Lebesgue almost all points.

\medskip

In this article, we want to further study the asymptotic behavior of of the  Birkhoff sum $S_n(x)$ of the Saint-Petersburg potential. We give a complete multifractal analysis of $S_n(x)$.

First, for any $\alpha\geq 1$, we consider the level set $$
E({\alpha})=\left\{x\in (0,1]: \lim_{n\to \infty}\frac{1}{n}S_n(x)=\alpha\right\}.
$$
For $t\in \mathbb{R}$ and $q>0$, define
\begin{eqnarray*}
  P(t,q):=\log \sum_{j=1}^{\infty} 2^{-tj-q(2^j-1)}.
\end{eqnarray*}
Then $P$ is a real-analytic function. Furthermore, for each $q>0$, there is a unique $t(q)>0$ such that $P(t(q), q)=0$. This function $q\mapsto t(q)$ is real-analytic, strictly decreasing and convex.

Denote by $\dim_H$ the Hausdorff dimension.
The function $\alpha\mapsto \dim_HE(\alpha)$, called the Birkhoff spectrum of the Saint-Petersburg potential $\varphi$, is  proved to be the Legendre transformation of the function $q \mapsto t(q)$.
\begin{thm}\label{main-1}
For any $\alpha \geq 1$ we have
\[
\dim_{H}E(\alpha)= \inf_{q>0} \{t(q)+q\alpha\}.
\]
Consequently, $\dim_HE(1)=0$ and the function $\alpha \mapsto \dim_{H}E(\alpha)$ is real-analytic,  strictly increasing, concave, and has limit $1$ as $\alpha \to \infty$.
\end{thm}

The Birkhoff spectrum of a continuous potential was obtained for full shifts \cite{Oli98}, for topologically mixing subshifts of finite type \cite{FFW01}, and for repellers of a topologically mixing $C^{1+\epsilon}$ expanding map \cite{BS01}.
A continuous potential in a compact space is bounded, hence these classical results are all for bounded potentials.
Our Theorem~\ref{main-1} gives a Birkhoff spectrum for an unbounded function with a singular point.
To prove Theorem~\ref{main-1}, we will transfer our question to a Birkhoff spectrum problem of an interval map with infinitely many branches and we will apply the techniques  developed in \cite{FLWW} for continued fraction dynamical system and in \cite{FJLR} for general expanding interval maps with infinitely many branches.

\medskip
We also study the Birkhoff sums $S_n(x)$ of fast increasing rates. Let $\Psi: \mathbb{N} \rightarrow \mathbb{N}$ be an increasing function.
For $\beta \in [0,\infty]$, consider the level set $$
E_\Psi({\beta}):=\left\{x\in (0,1]: \lim_{n\to \infty}\frac{1}{\Psi(n)}S_n(x)=\beta\right\}.
$$
\begin{thm}\label{t2}  
If $\Psi(n)$ is
one of the following
\[
\Psi(n)=n\log n, \ \Psi(n)=n^a \ (a>1), \ \Psi(n)=2^{n^{\gamma}} \ (0<\gamma<1/2),
\]
then for any $\beta \in [0,\infty]$, $\dim_HE_\Psi({\beta})=1$.

If $\Psi(n)=2^{n^{\gamma}}$ with $1/2\le \gamma<1$, then for any $\beta \in (0, \infty)$,
the set $E_\Psi({\beta})$ is empty, and $\dim_HE_\Psi({\beta})=1$ for $\beta=0, +\infty$.

If $\Psi(n)=2^{n^{\gamma}}$ with $\gamma\ge 1$, then for any $\beta  \in (0,\infty]$, the set $E_\Psi({\beta})$ is empty,  and $\dim_HE_\Psi({\beta})=1$ for $\beta=0$.
\end{thm}

We remark that by the above discussion on the convergence of $S_n(x)$, for all cases in Theorem \ref{t2}, the sets $E_\Psi(0)$ has full measure, and thus obviously has full Hausdorff dimension.

From the definition of $S_n(x)$, we see that for the integer $n$ such that $\epsilon_n(x)=1$, one has $S_n(x)=S_{n-1}(x)+1$, which implies for all $x$, $\liminf\limits_{n\to\infty} {S_{n}(x) \over S_{n-1}(x) }=1$. Thus if $\liminf\limits_{n\to\infty} {\Psi(n) \over \Psi(n-1) }>1$, then for any $\beta\in (0,\infty)$, the set $E_\Psi({\beta})$ is empty. By the definition of $S_n(x)$, we can also check that for all $x$, for the integer $n$ such that $\epsilon_n(x)=1$, we have $ S_n(x)\leq 2^{n}-1$. This implies $\liminf\limits_{n\to\infty} S_n(x)/2^n \leq 1$ (See also the formula (\ref{3.15}) in Section \ref{sec:fast}). Hence, for 'regular' growth functions $\Psi$ we only need to consider exponential and subexponential growth rates.

However, if we pick a point $x$ with dyadic expansion consisting mostly of $0$'s, with infinitely many $1$'s but in large distances from each other, then the Birkhoff sum $S_{n_i}(x)$ may grow arbitrarily fast on some subsequence $n_i$.
Thus for any increasing $\Psi(n)$ 
there exists a point $x$ such that $\limsup_{n \to \infty} S_n(x) / \Psi(n) = \infty$.

\medskip
Our study on the Saint-Petersburg potential is an attempt of multifractal analysis of unbounded potential functions on the doubling map dynamical system. However, the Saint-Petersburg potential is locally constant and not continuous. One might think of an another unbounded potential function $g: x\mapsto 1/x$ which is close to the Saint-Petersburg potential but is continuous.
In fact, our method for studying the fast increasing Birkhoff sum of Saint-Petersburg potential also works for the fast increasing Birkhoff sum of the  potential $g: x\mapsto 1/x$.

Denote by $S_ng(x)$ the Birkhoff sum
\[
 S_ng(x):=g(x)+g(T(x)) + \cdots+g(T^{n-1}(x)), \qquad  x\in (0,1].
\]
%
For $\beta \in [0,\infty]$, let $$
F_\Psi({\beta}):=\left\{x\in (0,1]: \lim_{n\to \infty}\frac{1}{\Psi(n)}S_ng(x)=\beta\right\}.
$$
\begin{thm}\label{t3}  
If $\Psi(n)$ is
one of the following
\[
\Psi(n)=n\log n, \ \Psi(n)=n^a \ (a>1), \ \Psi(n)=2^{n^{\gamma}} \ (0<\gamma<1/2),
\]
then for any $\beta  \in [0,\infty]$, $\dim_HF_\Psi({\beta})=1$.

If $\Psi(n)=2^{n^{\gamma}}$ with $1/2\le \gamma<1$, then for any $\beta  \in (0,\infty)$, the set $F_\Psi({\beta})$ is empty,  and $\dim_HF_\Psi({\beta})=1$ for $\beta=0, +\infty$.

If $\Psi(n)=2^{n^{\gamma}}$ with $\gamma\ge 1$, then for any $\beta  \in (0,\infty]$, the set $F_\Psi({\beta})$ is empty,  and $\dim_HF_\Psi({\beta})=1$ for $\beta=0$.

\end{thm}

We remark that these multifractal analysis on the Birkhoff sums of fast increasing rates have been done for some special potentials in continued fraction dynamical system (\cite{FLWW, LR1, LR2}).


%

%

\bigskip
\bigskip
\section{Birkhoff spectrum of the Saint-Petersburg potential}\label{sec:2}

In this section, we will obtain the Birkhoff spectrum of the Saint-Petersburg potential, i.e. the Hausdorff dimension of the following level set:
$$
E(\alpha):=\left\{x\in {(0,1]}: \lim_{n\to \infty}\frac{1}{n}\sum_{j=0}^{n-1}\varphi(T^jx)=\alpha\right\}\quad (\alpha\geq 1).
$$
We will transfer our question to a Birkhoff spectrum problem for an interval map with infinitely many branches.

\subsection{Transference lemma}

Recall that the Saint-Petersburg potential $\varphi$ is given by $$
\varphi(x)=2^{n}, \ {\rm{if}}\ \ x=[0^{n}1,\cdots]
$$where $x=[\epsilon_1\epsilon_2,\cdots]$ denotes the digit sequence in the binary expansion of $x$.
Recall also the definition of hitting time $n(x)$ and the acceleration $\widehat T$ of the doubling map $T$ in Section \ref{sec:intro}.
Define a new potential function
$$\phi(x):=2^{n(x)+1}-1, \ x\in (0,1].$$
In fact, $\phi$ is nothing but the function satisfying$$
\phi(x)=\sum_{j=0}^{n(x)}\varphi(T^jx).
$$
With the notation $n_1=n(x)+1\ge 1$, and $n_k=n(\widehat{T}^{k-1} x)+1$ for $k\geq 2$ given in Section \ref{sec:intro}, we have $$
\phi(\widehat{T}x)=\sum_{j=0}^{n(\widehat{T}x)}\varphi(T^j(\widehat{T}x))=\sum_{j=n_1}^{n_2-1}\varphi(T^jx).
$$
Hence,
\begin{equation}\label{f1}
\sum_{j=0}^{n_1+\cdots+n_{\ell}-1}\varphi(T^jx)=\sum_{k=0}^{\ell-1}\phi(\widehat{T}^kx)=2^{n_1}+\cdots+2^{n_{\ell}}-\ell.
\end{equation}
Note that the derivative of $\widehat{T}$ satisfies
\begin{align}\label{logTprime}
|\widehat{T}'|(x)=2^{n(x)+1}=2^{n_1}=\phi(x)+1.
\end{align}
We have
$$n_1+\cdots+n_{\ell}=\sum_{k=0}^{\ell-1}\log_2 |\widehat{T}'|(\widehat{T}^kx).$$

Recall the set in question: $$
E(\alpha)=\left\{x\in (0,1]: \lim_{n\to \infty}\frac{1}{n}\sum_{j=0}^{n-1}\varphi(T^jx)=\alpha\right\} \qquad (\alpha\geq 1).
$$
Define $$
\widetilde{E}(\alpha):=\left \{x\in (0,1]: \lim_{\ell\to \infty}\frac{\sum_{k=0}^{\ell-1}\phi(\widehat{T}^kx)}{\sum_{k=0}^{\ell-1}\log_2 |\widehat{T}'|(\widehat{T}^kx)}=\alpha \right\} \qquad (\alpha\geq 1).
$$

The following lemma shows the two level sets are the same.
\begin{lem}\label{lem:transference}
For all $\alpha \geq 1$, we have $E(\alpha)=\widetilde{E}(\alpha)$.
\end{lem}

\begin{proof}
It is evident that ${E(\alpha)\subset \widetilde{E}(\alpha)}$, because, as discussed above,
\begin{equation}\label{f2}
\frac{\sum_{k=0}^{\ell-1}\phi(\widehat{T}^kx)}{\sum_{k=0}^{\ell-1}\log_2  |\widehat{T}'|(\widehat{T}^kx)}=\frac{1}{n_1+\cdots+n_{\ell}}\sum_{j=0}^{n_1+\cdots+n_{\ell}-1}\varphi(T^jx).
\end{equation}

Now, we show the other direction. Take an $x\in \widetilde{E}(\alpha)$, express $x$ in its binary expansion
$$x=[0^{n_1-1} 10^{n_2-1} 1 \cdots 0^{n_{\ell}-1} 1\cdots].$$
In fact, $n_{\ell}-1$ is the recurrence time for $n(\widehat{T}^{{\ell}-1}x)$, for each ${\ell}\ge 1$.

By \eqref{f1}, we have, at present, $$
\lim_{\ell\to \infty}\frac{1}{n_1+\cdots+n_{\ell}}\sum_{j=0}^{n_1+\cdots+n_{\ell}-1}\varphi(T^jx)=\alpha.
$$
So, we are required to check it holds for all $n$.

For any $\epsilon>0$, there exists ${\ell}_0\in \N$ such that, for any ${\ell}\ge {\ell}_0$,
\begin{equation}\label{f3}
\alpha-\epsilon\le \frac{2^{n_1}+\cdots+2^{n_{\ell}}-{\ell}}{n_1+\cdots+n_{\ell}}\le \alpha+\epsilon.
\end{equation}

For any $n_1+\cdots+ n_{\ell}< n<n_1+\cdots+n_{\ell}+n_{{\ell}+1}$ with ${\ell}\ge {\ell}_0$, it is trivial that $$
\frac{2^{n_1}+\cdots+2^{n_{\ell}}-{\ell}}{n_1+\cdots+n_{\ell}+n_{{\ell}+1}}\le \frac{1}{n}\sum_{j=0}^{n-1}\varphi(T^jx)\le \frac{2^{n_1}+\cdots+2^{n_{\ell}}+2^{n_{{\ell}+1}}-{\ell}-1}{n_1+\cdots+n_{\ell}}.
$$ Thus, it suffices to show that
\begin{equation}\label{n-ell-small}
2^{n_{{\ell}+1}}=o(n_1+\cdots+n_{\ell}), 
\end{equation}
which also implies
\[
{n_{{\ell}+1}}=o(n_1+\cdots+n_{\ell}).
\]
Let $M_0$ be a large integer such that, for all $M\ge M_0$, $2^{M}\ge 4\alpha M$. So, when $n_{{\ell}+1}\le M_0$, there is nothing to prove. So, we always assume $2^{n_{{\ell}+1}}\ge 4\alpha n_{{\ell}+1}$.

By \eqref{f3}, we have
\begin{align*}
2^{n_1}+\cdots+2^{n_{\ell}}-{\ell} &\ge (\alpha-\epsilon)(n_1+\cdots+n_{\ell}), \\
2^{n_1}+\cdots+2^{n_{\ell}}+2^{n_{{\ell}+1}}-{\ell}-1&\le (\alpha+\epsilon)(n_1+\cdots+n_{\ell}+n_{{\ell}+1}).
\end{align*}
These give $$
2^{n_{{\ell}+1}}\le 2\epsilon(n_1+\cdots+n_{\ell})+(\alpha+\epsilon)n_{{\ell}+1}+1.
$$
So, we have
\begin{equation*}
2^{n_{{\ell}+1}}\le 4\epsilon (n_1+\cdots+n_{\ell}).  \qedhere
\end{equation*}
\end{proof}

\subsection{Dimension of $\widetilde{E}(\alpha)$}

Now we calculate the Hausdorff dimension of the set $\widetilde{E}(\alpha)$. At first, we give a notation.
\begin{itemize}
\item For each finite word $w\in \bigcup_{n\ge 1}\{0,1\}^n$ of length $n$, a $T$-{\it dyadic} cylinder of order $n$ is defined as $$
I_n(w)=\{x\in (0,1]: (\epsilon_1 (x),\cdots,\epsilon_n(x))=w\}.
$$

\item For $(n_1,\cdots,n_{\ell})\in (\N\setminus \{0\})^{\ell}$, a $\widehat{T}$-{\it dyadic} cylinder of order $\ell$ is defined as $$
D_{\ell}(n_1,\cdots,n_{\ell})=\{x\in (0,1]: n_k(x)=n_k, \ 1\le k\le \ell\}.
$$ \end{itemize}
\begin{proof}[Proof of Theorem \ref{main-1}]  
To calculate the Hausdorff dimension of $\widetilde{E}(\alpha)$, we construct a suitable measure supported on $\widetilde{E}(\alpha)$. The Gibbs measures derived from the Ruelle-Perron-Frobenius transfer operator are good candidates for such a measure.

In fact, by considering the inverse branches $U_i: x\mapsto {x+1 \over 2^{i+1}} \ (i\geq 0)$ of $\widehat{T}$, we can code the dynamical system $([0,1], \widehat{T})$ by the conformal infinite iterated function system $(U_i)_{i\geq 0}$ which satisfies the open set condition (\cite[Section 1]{HMU}).

Consider the potential function with two parameters $$\psi_{t,q}:=-t\log |\widehat{T}'|-(\log 2)\cdot q\phi \quad ( t \in \mathbb{R}, \ q>0).$$
Then $(\psi_{t,q}\circ U_i)_{i\geq 0}$ is a family of strong H\"older family (\cite[Page 30]{HMU}).
Hence, we can define a Ruelle operator
\[
    \mathcal{L}_{t,q} f(x):= \sum_{y\in \widehat{T}^{-1}x} e^{\psi_{t,q}(y)} f(y),
\]
on the Banach space of continuous functions on the corresponding infinite symbolic space (\cite[Page 31]{HMU}).

By the Ruelle-Perron-Frobenius transfer operator theory \cite[Theorems 2.9 and 2.10]{HMU}, for any $q>0$ (to satisfy the condition 2.2 of \cite{HMU}), we can find an eigenvalue $\lambda_{t,q}$ and an eigenfunction $h_{t,q}$ for $\mathcal{L}_{t,q}$ and an eigenfunction $\nu_{t,q}$ for the conjugate operator $\mathcal{L}_{t,q}^*$. Then the pressure function $P(t,q)=\log \lambda_{t,q}$ and the ergodic Gibbs measure $\mu_{t,q}$ is given by
$h_{t,q}\cdot \nu_{t,q}$.

The pressure function can be computed by (see \cite[Pages 31 and 48]{HMU})
\begin{eqnarray*}
  P(t,q)= \lim_{\ell\to\infty}\frac{1}{\ell} \log \left(\sum_{(n_1, \cdots, n_\ell) \in (\N\setminus \{0\})^{\ell}} \exp \sup_{x\in D_{\ell}(n_1,\cdots,n_{\ell})}(S_{\ell}\psi_{t,q} (x))\right).
\end{eqnarray*}
Note that for all $n\geq 1$,
\[
|\widehat{T}'|(x)=2^{n}, \ \text{and} \ \phi(x)=2^{n}-1\ \  \text{if} \ x\in \left({1\over 2^{n}}, {1\over 2^{n-1}}\right]=D_1(n).
\]
Then for $x\in D_{\ell}(n_1,\cdots,n_{\ell})$,
\[
S_{\ell}\psi_{t,q} (x)= -(\log 2) \cdot t \cdot \sum_{j=1}^\ell n_j +(\log 2) \cdot q \cdot \sum_{j=1}^\ell (2^{n_j}-1).
\]
Thus
\begin{eqnarray*}
& & \sum_{(n_1, \cdots, n_\ell) \in (\N\setminus \{0\})^{\ell}} \exp \sup_{x\in D_{\ell}(n_1,\cdots,n_{\ell})}(S_{\ell}\psi_{t,q} (x))\\
&=&  \sum_{n_1=1}^\infty\cdots \sum_{n_\ell=1}^\infty \left( \prod_{j=1}^\ell 2^{-tn_j} \cdot \prod_{j=1}^\ell 2^{q(2^{n_j}-1)}\right)=\left(\sum_{j=1}^{\infty} 2^{-tj-q(2^j-1)}\right)^\ell.
\end{eqnarray*}
Hence
\begin{eqnarray*}
  P(t,q)= \log \sum_{j=1}^{\infty} 2^{-tj-q(2^j-1)}.
\end{eqnarray*}

Now we calculate the local dimension of the Gibbs measure $\mu_{t,q}$. Let $D_{\ell}(x)$ be the $\widehat T$-dyadic cylinder containing $x$ of order $\ell$.
By the Gibbs property of $\mu_{t,q}$,
\begin{align}\label{ff1}
\begin{split}
   \frac{\log \mu_{t,q} (D_{\ell}(x))}{\log |D_{\ell}(x)|}&=\frac{S_{\ell}\psi_{t,q}(x)- \ell P(t,q)}{-S_{\ell}\log |\widehat{T}'|(x)}\\
   &=\frac{-tS_{\ell}\log |\widehat{T}'|(x)-(\log 2)\cdot q S_{\ell}\phi(x)-\ell P(t,q)}{-S_{\ell}\log |\widehat{T}'|(x)}\\
   &= t + q\frac{S_{\ell}\phi(x)}{S_{\ell}\log_2 |\widehat{T}'|(x)}+ \frac{\ell P(t,q)}{S_{\ell}\log |\widehat{T}'|(x)}.
   \end{split}
\end{align}

$\bullet$ {\sc Upper bound.}  For each $q>0$, let $t(q)$ be the the number such that $P(t(q), q)=0$.  (The existence of $t(q)$ comes from the facts that $P(t,q)$ is real-analytic and that for fixed $q>0$, $P(t,q)>0$ when $t\to -\infty$ and $P(t,q)<0$ when $t\to +\infty$.)  Then for all $x\in \widetilde{E}(\alpha)$, we have
\begin{align*}
   \liminf_{r\to 0} {\log \mu_{t,q}(B(x,r)) \over \log r} &\le \liminf_{\ell\to \infty}\frac{\log \mu_{t,q}(B(x, |D_{\ell}(x)|))}{\log |D_{\ell}(x)|}\\&\leq \liminf_{\ell\to \infty}\frac{\log \mu_{t,q} (D_{\ell}(x))}{\log |D_{\ell}(x)|}=t(q) + q{\alpha},
\end{align*} where for the second inequality the trivial inclusion $D_{\ell}(x)\subset B(x, |D_{\ell}(x)|)$ is used.
By Billingsley Lemma (see e.g. \cite[Proposition 4.9.]{Fa1}), this gives an upper bound of the Hausdorff dimension of $\widetilde{E}(\alpha)$. Thus we have
\[
\dim_H \widetilde{E}(\alpha) \leq \inf_{q>0} \{t(q)+q{\alpha}\}.
\]


$\bullet$ {\sc Lower bound.}  By the real-analyticity of $P(t,q)$ and the implicit function theorem, the function $q\mapsto t(q)$ is also real-analytic. Thus there exists $q_0$ such that the following infimum is attained
\[
   \inf _{q>0}\{ t(q)+q{\alpha}\}.
\]
Then we have \begin{equation}\label{t'} t'(q_0)+\alpha=0.\end{equation}

To prove the lower bound, we first show two claims.

\smallskip

Claim (A): The measure $\mu_{t(q_0), q_0}$ is supported on $E_{\alpha}$.

\smallskip
On the one hand, since $P(t(q), q)=0$,
\begin{equation}\label{derivative-q}
\frac{\partial P}{\partial t} t'(q)+\frac{\partial P}{\partial q}=0.
\end{equation}
On the other hand, by the ergodicity of the measure $\mu_{t,q}$, we have for $\mu_{t,q}$ almost all $x$,
\[
\lim_{\ell\to\infty}\frac{S_{\ell}\phi(x)}{S_{\ell}\log_2 |\widehat{T}'|(x)}= \frac{\int\phi d\mu_{t,q}}{\int \log |\widehat{T}'| d\mu_{t,q} }\cdot \log 2.
\]
By Ruelle-Perron-Frobenius transfer operator theory (\cite{HMU}, Proposition 6.5),
\[\int (\log 2) \cdot \phi d\mu_{t,q}= -\frac{\partial P}{\partial q} \quad \text{and} \quad\int \log |\widehat{T}'| d\mu_{t,q} = -\frac{\partial P}{\partial t}.\]
Thus by (\ref{derivative-q}) and then  (\ref{t'}), for $\mu_{t(q_0), q_0}$ almost all $x$,
\[
\lim_{\ell\to\infty}\frac{S_{\ell}\phi(x)}{S_{\ell}\log_2 |\widehat{T}'|(x)}=\frac{\frac{\partial P}{\partial q}}{\frac{\partial P}{\partial t}}=-t'(q_0)=\alpha.
\]
This shows Claim (A).

\smallskip
Claim (B): For $\mu_{t(q_0), q_0}$ almost all $x$,
\[
  \lim_{n\to\infty} {\log \mu_{t(q_0),q_0} (I_n(x)) \over \log 2^{-n}}=t(q_0)+q_0 \alpha,
\]
where $I_n(x)$ is the $T$-dyadic cylinder of order $n$ containing $x$.

\smallskip

On the one hand, by \eqref{ff1} and then by (\ref{derivative-q}) and \eqref{t'}, one has for $\mu_{t(q_0),q_0}$ almost all $x$
\begin{equation}\label{local-basic}
\lim_{\ell\to \infty}\frac{\log \mu_{t(q_0),q_0}(D_{\ell}(x))}{\log |D_{\ell}(x)|}=t(q_0) + q_0\frac{\frac{\partial P}{\partial q}}{\frac{\partial P}{\partial t}}=t(q_0)+q_0 \alpha.
\end{equation}
On the other hand, note that for any $x\in E(\alpha)$, if the binary expansion of $x$ is $x=[0^{n_1-1}10^{n_2-1}1\dots]$, then for any $\delta>0$, for $\ell$ large enough,
\[
(\alpha-\delta)\ell \leq 2^{n_1}+\cdots+ 2^{n_\ell}-\ell=S_\ell\phi(x) \leq (\alpha + \delta) \ell.
\]
Hence 
\[
n_\ell = O(\log \ell),
\]
which implies
\begin{equation}\label{length-basic}
\lim_{\ell\to \infty}{\log |D_\ell(x)| \over \log |D_{\ell+1}(x)|}= \lim_{\ell\to \infty} { {n_1+\cdots +{n_\ell}} \over {n_1+\cdots +{n_\ell}+{n_{\ell+1}}}}=1.
\end{equation}
Thus \[
  \lim_{n\to\infty} {\log \mu_{t(q_0),q_0} (I_n(x)) \over \log 2^{-n}} =\lim_{\ell\to \infty}\frac{\log \mu_{t(q_0),q_0}(D_{\ell}(x))}{\log |D_{\ell}(x)|}.
\]
This shows Claim  (B).

To conclude the desired lower bound, we apply the classical mass distribution principle (see \cite[Proposition 4.2]{Fa1}).
Since the Hausdorff dimension will not be changed if we replace the $\delta$-coverings by $T$-dyadic cylinder coverings (see \cite[Section 2.4]{Fa1}),
the lower bound of the Hausdorff dimension can be given by the mass transference principle on $T$-dyadic cylinders.
By the above two claims and Egorov's theorem, for any $\eta>0$, there exists an integer $N_0$ such that the set $$
\Big\{x\in E_{\alpha}: \mu(I_n(x))\le |I_n(x)|^{t(q_0)+q_0 \alpha-\eta}, \ n\ge N\Big\} $$
is of $\mu_{t(q_0),q_0}$ positive measure. So, it implies that $$
\dim_HE_{\alpha}\ge t(q_0)+q_0 \alpha-\eta.
$$



Note that (\cite[Lemma 7.5]{HMU}) the function $q\mapsto t(q)$ is a decreasing convex function 
such that
\[t(0)=1, \quad \lim_{q\to\infty} (t(q) +q)=0,\]
and
\[
\lim_{q\to 0^+} t'(q)=-\infty, \quad \lim_{q\to+\infty}t'(q)=-1.
\]
Therefore, we have proved for any $\alpha \in (1, +\infty)$
\[
\dim_{H}(\widetilde{E}(\alpha))= \inf_{q>0} \{t(q)+q\alpha\},
\]
which is Legendre transformation.
All the properties stated in Theorem~\ref{main-1} are satisfied by the function $\alpha \mapsto \dim_{H}(\widetilde{E}(\alpha))$ which is the same function as $\alpha \mapsto \dim_{H}({E}(\alpha))$ by Lemma~\ref{lem:transference}.

For the end point $\alpha=1$, it suffices to note that the level set $E(1)$ is nothing but the set of numbers with frequency of the digit $1$ in its binary expansion being $1$.
Thus the Hausdorff dimension of $E(1)$ is $0$. Hence, the Legendre transformation formula for the Hausdorff dimension of $E(\alpha)$ ($\alpha>1$)  also holds for $\alpha=1$.
\end{proof}

\section{Fast increasing Birkhoff sum}\label{sec:fast}

At first, we give two simple observations.

\begin{lem}\label{lem:1}
Let $W$ be an integer such that $2^t\le W<2^{t+1}$ for some positive integer $t$. For any $0\leq n\leq t$, among the integers between $W$ and $W(1+2^{-n})$, there is one $V=V(W,n)$ whose binary expansion of $V$ has at most $n+2$ digits $1$ and ends with at least $t-n$ zeros.
\end{lem}
\begin{proof}
By the assumption, we have
$2^{-n}W \geq 2^{t-n} 
$.
Thus among the $2^{-n}W$ consecutive integers from $W$ to $W(1+2^{-n})$ there is at least one integer which is divisible by $2^{t-n}$ which means there is an integer $\ell\geq 1$ such that $$W \leq \ell 2^{k-n} \leq W(1+2^{-n}).$$
Let $V=\ell 2^{t-n}$ and note that $V$ is an integer whose binary expansion ends with at least $t-n$ zeros. Since $\ell 2^{t-n} \leq W(1+2^{-n}) < 2^{t+2}$, we conclude that $\ell 2^{t-n}$ has at most $(t+2)-(t-n)=n+2$ digits $1$ in its binary expansion.
\end{proof}

In the follows, the base of the logarithm is taken to be $2$.
\begin{lem}\label{lem:2}
For each integer $W$, and any integer $n\leq \log W$, we can find a word $w$ with length
\[
|w| \leq (n+2)(2+\log W)
\]
and for any $x\in I_{|w|}(w)$
 $$
W \leq \sum_{j=0}^{|w|-1}\varphi(T^jx) \leq W(1+2^{-n}).
$$
\end{lem}
\begin{proof}
Let $V$ be an integer given in Lemma \ref{lem:1}. Then 
$W \leq V \leq W(1+2^{-n}).$ Moreover if we write this number $V$ in binary expansion:
\[
V=2^{t_1}+\cdots+ 2^{t_p},
\]
one has that $ \lfloor \log W\rfloor+1\geq t_1>\cdots>t_p\ge \lfloor \log W\rfloor -n$ and $p\le n+2$.
Consider the word $$
w=(10^{t_1-1}1, 1 0^{t_2-1}1, \cdots, 10^{t_p-1}1)
$$ here the word $10^{t_p-1}1$ is $1$ when $t_p=0$.
 Then we can check that the length of $w$ satisfies $$
|w|=(t_1+1)+\cdots+(t_p+1)\le p(t_1+1)\le (n+2) (2+\log W),
$$  and for any $x\in I_{|w|}(w)$, $$
\sum_{j=0}^{|w|-1}\varphi(T^jx)=V.
$$
Hence, the proof is completed.
\end{proof}

We also need the following lemmas whose proofs are left for the reader.
\begin{lem}\label{lem:3}
For any $m\geq 1$, define $$
F_m=\big\{x\in (0,1]: \epsilon_{km}(x)=1, \ {\text{for all}}\ k\ge 1\big\}.
$$ Then $\dim_HF_m=\frac{m-1}{m}$.
\end{lem}

\begin{lem}\cite[Lemma 4]{MaWen}\label{l1} 
Given a subset $\mathbb{J}$ of positive integers and an
infinite sequence $\{a_k\}_{k=1}^{\infty}$ of 0's and 1's, let
$$E(\mathbb{J}, \{a_k\}_{k=1}^{\infty})=\Big\{x\in (0,1]: \epsilon_k(x)=a_k, {\text{for all}}\ k\in \mathbb{J}\Big\}.$$
If the density of $\mathbb{J}$ is zero, that is,
$$\lim_{n\to \infty}\frac{1}{n}{\text{Card}}\Big\{k\le n: k\in \mathbb{J}\Big\}=0$$
then $\dim_H E(\mathbb{J}, \{a_k\}_{k=1}^{\infty})=1$.
\end{lem}

Before the proof Theorem \ref{t2}, we show the following lemma.
\begin{lem}\label{l2}
Let $\Psi: \mathbb{N} \rightarrow \mathbb{N}$ be an increasing function such that $\Psi(n)/n\to \infty$ as $n\to\infty$. Assume that there exists a subsequence $N_k$ satisfying the following conditions
\begin{eqnarray}\label{f6}N_{k}-N_{k-1}\to \infty, \ \Psi(N_k)-\Psi(N_{k-1})\to \infty,\end{eqnarray} and \begin{eqnarray}\label{f6'}
\ \frac{\Psi({N_{k-1}})}{\Psi({N_{k}})}\to 1, \ \ \frac{\log \big(\Psi(N_{k})-\Psi(N_{k-1})\big)}{N_{k}-N_{k-1}}\to 0,
\end{eqnarray}
as $k\to \infty$. Then the set $$
E_\Psi (1)=\Big\{x\in (0,1]: \lim_{n\to \infty}\frac{1}{\Psi(n)}S_n(x)=1\Big\}
$$
has Hausdorff dimension $1$.
\end{lem}
\noindent $Proof.$ Fix a large integer $m$ and write $$
\mathcal{U}=\Big\{u=(\epsilon_1,\cdots,\epsilon_m): \epsilon_m=1, \epsilon_i\in \{0,1\}, i\ne m\Big\}.
$$To avoid the abuse of notation, by the first assumption of (\ref{f6}), we assume $N_{k}-N_{k-1}\gg m$ for all $k\ge 1$ by setting $N_0=0$ and $\Psi(N_0)=0$.

For each $k\ge 1$,  we write $$W_k:=\Psi(N_k)-\Psi(N_{k-1})$$
and let $\{n_k\}$ be a sequence of integers tending to $\infty$ such that
\[ n_k\le \log W_k,\ \
n_k \cdot \frac{\log \big(\Psi(N_{k})-\Psi(N_{k-1})\big)}{N_{k}-N_{k-1}}\to 0.
\]
By the second assumptions of (\ref{f6}) and (\ref{f6'}), this sequence of $n_k\geq 0$ do exist.

Now for $W_k$ and $n_k$, let $w_k$ be the word given in Lemma \ref{lem:2}. Then 
the length $a_k$ of $w_k$ satisfies
\begin{align}\label{f5}
\begin{split}
a_k\le& (n_k+2)(2+\log W_k)\\
=&(n_k+2)\left(2+\log(\Psi(N_k)-\Psi(N_{k-1}))\right)=o(N_{k}-N_{k-1})
\end{split}
\end{align}
and for any $x\in I_{a_k}(w_k)$,
\begin{eqnarray}\label{f4}
W_k \leq \sum_{j=0}^{a_k-1}\varphi(T^jx) \leq W_k (1+2^{-n_k}).
\end{eqnarray}

Define $t_k, \ell_k$ to be the integers satisfying $$
N_k-N_{k-1}-a_{k}=t_km+\ell_k,
$$ for some $0\le \ell_k<m$.

Let $w_k\ (k\geq 1)$ be given as the above. We define a Cantor subset of $E_{\Psi}(1)$ as follows.

\medskip

{\it Level 1 of the Cantor subset.}  Define
\begin{eqnarray*}
E_1=\Big\{I_{N_1}(u_1,\cdots,u_{t_1}, 1^{\ell_1}, w_1): u_i\in \mathcal{U}, 1\le i\le t_1\Big\}.
\end{eqnarray*} 

For simplicity, we use $I_{N_1}(U_1)$ to denote a general cylinder in $E_1$.

\medskip
{\it Level 2 of the Cantor subset.}
 This level is composed by sublevels for each cylinder $I_{N_1}(U_1)\in E_1$. Fix an element $I_{N_1}=I_{N_1}(U_1)\in E_1$. Define
\begin{eqnarray*}
E_2(I_{N_1}(U_1))=\Big\{I_{N_2}(U_1,u_1,\cdots,u_{t_2}, 1^{\ell_2}, w_2): u_i\in \mathcal{U}, 1\le i\le t_2\Big\}.
\end{eqnarray*}
Then $$
E_2=\bigcup_{I_{N_1}\in E_1}E_2(I_{N_1}).
$$
For simplicity, we use $I_{N_2}(U_2)$ to denote a general cylinder in $E_2$.
\medskip

{\it From Level $k$ to $k+1$.}
 Fix $I_{N_k}(U_k)\in E_k$. Define
\begin{eqnarray*}
E_{k+1}(I_{N_k}(U_k))=\Big\{I_{N_{k+1}}(U_k,u_1,\cdots,u_{t_{{k+1}}}, 1^{\ell_{k+1}}, w_{k+1}): u_i\in \mathcal{U}, 1\le i\le t_{k+1}\Big\}.
\end{eqnarray*}
Then $$
E_{k+1}=\bigcup_{I_{N_k}\in E_k}E_{k+1}(I_{N_k}).
$$
%

Up to now we have constructed a sequence of nested sets $\{E_k\}_{k\ge 1}$. Set $$
F=\bigcap_{k\ge 1}E_k.
$$
We claim that $$ F\subset E(\Psi).
$$
In fact, for all $x\in F$, by construction, for each $k\ge 1$, \begin{align*}
&\sum_{n=N_{k-1}}^{N_{k}-1}\varphi(T^nx)\\=&\sum_{n=N_{k-1}}^{N_{k-1}+t_{k}m-1}\varphi(T^nx)+\sum_{n=N_{k-1}+t_{k}m}^{N_{k-1}+t_{k}m+\ell_{k}-1}\varphi(T^nx)+\sum_{n=N_{k-1}+t_{k}m+\ell_{k}}^{N_{k}-1}\varphi(T^nx)\\
=&t_{k}O(2^m)+\ell_{k}+W_{k}(1+O(2^{-n_{k}}))\\=&O\left(\frac{(N_{k}-N_{k-1})2^m}{m}\right)+(\Psi(N_{k})-\Psi(N_{k-1}))(1+O(2^{-n_{k}})).
\end{align*}
Since $n_k\to\infty$ which implies $2^{-n_{k}}\to0$ as $k\to\infty$, we have
\begin{align*}
\sum_{n=0}^{N_{k}-1}\varphi(T^nx)&=\Psi(N_{k}) \big(1+o(1)\big)+O\left(\frac{N_{k}2^m}{m}\right).
\end{align*}
By the assumption $\Psi(n)/n\to \infty$ as $n\to \infty$, we then deduce
\begin{align*}
\sum_{n=0}^{N_{k}-1}\varphi(T^nx)=\Psi(N_{k})+o(\Psi(N_{k})),
\end{align*}
Thus
 \begin{eqnarray}\label{N-k-limit}
\lim_{k\to \infty}{\sum_{n=0}^{N_{k}-1}\varphi(T^nx) \over \Psi(N_{k})}=1.
\end{eqnarray}
While, for each $N_{k-1}<N\le N_{k}$
\begin{eqnarray*}
\frac{\sum_{n=0}^{N_{k-1}-1}\varphi(T^nx)}{\Psi(N_{k})}\le \frac{\sum_{n=0}^{N-1}\varphi(T^nx)}{\Psi(N)}\le \frac{\sum_{n=0}^{N_{k}-1}\varphi(T^nx)}{\Psi(N_{k-1})}.
\end{eqnarray*}
So by the first assumption of (\ref{f6'}), we deduce from (\ref{N-k-limit}) that
$$
\lim_{n\to\infty}\frac{1}{\Psi(n)}S_n(x)=1.
$$
This proves $x\in E_\Psi(1)$ and hence $F\subset E_\Psi(1)$.

\bigskip
In the following, we will construct a H\"{o}lder continuous function from $F$ to $F_m$.
Recall that
$$
F_m=\big\{x\in (0,1]: \epsilon_{km}(x)=1, \ {\text{for all}}\ k\ge 1\big\}.
$$
Define \begin{eqnarray*}
f:F&\to& F_m\\
x&\mapsto & y
\end{eqnarray*}where $y$ is obtained by eliminating the digits $\{(\epsilon_{N_k-\ell_k-a_k+1}, \cdots,\epsilon_{N_k})\}_{k\ge 1}$ in the binary expansion of $x$. Now we calculate the H\"{o}lder expoent of $f$.

Take two points $x_1, x_2\in F$ closed enough. Let $n$ be the smallest integer such that $\epsilon_{n}(x_1)\neq \epsilon_n(x_2)$ and $k$ be the integer such that $N_k<n\le N_{k+1}$. Note that by the construction of $F$, the digits sequence $$\{(\epsilon_{N_k-\ell_k-a_k+1}, \cdots,\epsilon_{N_k})\}_{k\ge 1}\ {\text{and}}\ \{\epsilon_{N_k+tm}\}_{1\le t\le t_{k+1}}$$ are the same for all $x\in F$. So we must have
\begin{eqnarray}\label{f7}N_k<n< N_{k+1}-\ell_{k+1}-a_{k+1}.\end{eqnarray} Since  $n$ is strictly less than $N_{k+1}-\ell_{k+1}-a_{k+1}$ and $\epsilon_{N_k+tm}(x_1)=\epsilon_{N_k+tm}(x_2)=1$ for all $1\le t\le  t_{k+1}$, thus, at most $m$ steps after the position $n$, saying $n'$, $\epsilon_{n'}(x_1)=\epsilon_{n'}(x_2)=1$. So it follows that
$$
|x_1-x_2|\ge \frac{1}{2^{n+m}}.
$$


Again by the construction and the definition of the map $f$, we have $y_1=f(x_1)$ and $y_2=f(x_2)$ have common digits up to the position $n-1-(\ell_1+a_1)-\cdots-(\ell_k+a_k)$. Thus, it follows
$$
|f(x_1)-f(x_2)|\le \frac{1}{2^{n-1-(\ell_1+a_1)-\cdots-(\ell_k+a_k)}}.
$$
Recall that $\ell_k<m$ and $a_1+\cdots+a_k=o(N_k)$ (see (\ref{f5})) and also that $N_k/k\to \infty$ as $k\to \infty$ (by (\ref{f6})). We have
$$
1\ge \frac{n-1-(\ell_1+a_1)-\cdots-(\ell_k+a_k)}{n+m}\ge \frac{n-1-km-o(N_k)}{n+m}=1+o(1),
$$
which implies that $f$ is $(1-\eta)$-H\"{o}lder for any $\eta>0$. Thus
$$
\dim_HF\ge (1-\eta)\dim_HF_m .
$$
By Lemma \ref{lem:3}, we then have
$$
\dim_HF\ge (1-\eta)\frac{ m-1}{ m}.
$$
By the arbitrariness of $\eta>0$ and letting $m\to \infty$, we conclude that $\dim_HE(\Psi)=1$. This finishes the proof.\hfill $\Box$

\begin{proof}[Proof  of  Theorem~\ref{t2}]
In all the three parts of Theorem \ref{t2}, the case of $\beta= 0$ is a direct consequence of Theorem~\ref{main-1}.

(I). Assume that $\Psi$ is one of the functions $\Psi(n)=n\log n$, $\Psi(n)=n^a\ (a>1)$, $\Psi(n)=2^{n^\gamma}$ with $0<\gamma<1/2$.

(I$_1$). $0<\beta<\infty$. It suffices to consider the dimension of $E_\Psi(1)$ i.e. $\beta=1$, since for other $\beta\in (0,\infty)$,
 we need only replace $\Psi(n)$ by $\beta\Psi(n)$.
 
%

To show $\dim_H E_\Psi(1)=1$, we can apply Lemma \ref{l2} directly.
If $\Psi(n)=n\log n$, we can choose $N_k=k^2$. For $\Psi(n)=n^a\ (a>1)$, we can also choose $N_k=k^2$.
Suppose now $\Psi(n)=2^{n^\gamma}$ with $0<\gamma<1/2$. Let $\delta>0$ be small such that \begin{equation}\label{ff2}
\frac{\gamma}{1-\gamma}+\delta \gamma<1\end{equation} which is possible since $\gamma<1/2$.
Take
\begin{equation}\label{cond1-OK}
N_k=\lfloor k^{{1\over 1-\gamma}+\delta}\rfloor.
\end{equation}
Then we have
\begin{equation}\label{cond2-OK}
N_{k+1}-N_k \approx k^{{\gamma \over 1-\gamma}+\delta},
\end{equation}
and
\begin{align*}
\log(\Psi(N_{k+1})-\Psi(N_{k})) &\approx \log(\Psi'(N_k) (N_{k+1}-N_k)) \\  &\approx N_k^{\gamma} + \log(N_{k+1}-N_k)\approx N_k^\gamma.
\end{align*}
Here we write $A \approx B$ when $A/B \to 1$.
This shows the validity of \eqref{f6}. Moreover,
\[
{\log(\Psi(N_{k+1})-\Psi(N_{k})) \over N_{k+1}-N_k} \approx {k^{{\gamma \over 1-\gamma}+\gamma\delta} \over k^{{\gamma \over 1-\gamma}+\delta}}=k^{-\delta(1-\gamma)} \to 0 \ (k\to \infty).
\]
Thus the second assumption of (\ref{f6'}) is satisfied. At last, for the first assumption in (\ref{f6'}), by (\ref{ff2})$$
\frac{\Psi(N_{k-1})}{\Psi(N_k)}=2^{(k-1)^{\frac{\gamma}{1-\gamma}+\delta \gamma}-k^{\frac{\gamma}{1-\gamma}+\delta \gamma}}\to 1.
$$
Hence Lemma \ref{l2} applies.

\medskip

(I$_2$). If $\beta = \infty$, 
we may choose 
$\widetilde{\Psi}(n)=2^{n^{\eta}}$ for some $0 < \eta < \frac 12$ such that
$E_{\widetilde{\Psi}} (1) \subset E_\Psi (\infty)$. Then $\dim_HE_\Psi (\infty)=1$ follows from (I$_1$).


\medskip

(II). Now suppose that $\Psi(n)=2^{n^{\gamma}}$ with $1/2 \leq \gamma < 1$.

(II$_1$). Let $\beta \in (0, \infty)$.
We will prove that $E_\Psi({\beta})$ is empty. On the contrary,
suppose there is $x \in E_\Psi({\beta})$, which has binary expansion
\begin{equation}\label{3.3}x=[0^{n_1-1} 10^{n_2-1} 1 \cdots 0^{n_{\ell}-1} 1\cdots].\end{equation}
 Then, by \eqref{f1} we have
\begin{equation}\label{3.1}
\begin{split}
\frac{S_{n_1 + n_2 +\dots + n_{\ell}}(x)}{\Psi(n_1 + n_2 +\dots + n_{\ell})} &= \frac{2^{n_1} + 2^{n_2} + \dots + 2^{n_{\ell}} - {\ell}}{2^{(n_1 + n_2 + \dots + n_{\ell})^\gamma}} \to \beta, \\
\frac{S_{n_1 + n_2 +\dots + n_{\ell}+1}(x)}{\Psi(n_1 + n_2 +\dots + n_{\ell}+1)} &= \frac{2^{n_1} + 2^{n_2} + \dots + 2^{n_{\ell}} -{\ell} + 2^{n_{{\ell}+1}-1}}{2^{(n_1 + n_2 + \dots + n_{\ell} +1)^\gamma}} \to \beta.
\end{split}\end{equation}
Since
$$\frac{2^{(n_1 + n_2 + \dots + n_{\ell})^\gamma}}{2^{(n_1 + n_2 + \dots + n_{\ell}+1)^\gamma}} \to 1,$$
by dividing the two limits of \eqref{3.1}, we deduce that
\begin{equation*}
\frac{2^{n_1} + 2^{n_2} + \dots + 2^{n_{\ell}} -{\ell} + 2^{n_{{\ell}+1}-1}}{2^{n_1} + 2^{n_2} + \dots + 2^{n_{\ell}} - {\ell} }  = 1 +\frac{2^{n_{{\ell}+1}-1}}{2^{n_1} + 2^{n_2} + \dots + 2^{n_{\ell}}-{\ell}}  \to 1,
\end{equation*}
which implies that
\begin{equation*}\label{SnSn1}
\frac{S_{n_1 + n_2 +\dots + n_{{\ell}+1}}(x)}{{S_{n_1 + n_2 +\dots + n_{\ell}}(x)}}=1+\frac{2^{n_{{\ell}+1}}-1}{2^{n_1} + 2^{n_2} + \dots + 2^{n_{\ell}}-{\ell}}  \to 1.
\end{equation*}
Combining with \eqref{3.1}, we get
\begin{equation*}
1\leftarrow \frac{\Psi(n_1+\cdots+n_{{\ell}+1})}{\Psi(n_1+\cdots+n_{\ell})}=\frac{2^{(n_1 + n_2 + \dots + n_{\ell} + n_{{\ell} +1})^\gamma}}{2^{(n_1 + n_2 + \dots + n_{\ell})^\gamma}}.
\end{equation*}
Thus
\begin{align*}
&(n_1 + n_2 + \dots + n_{\ell}+ n_{{\ell} +1})^\gamma - (n_1 + n_2 + \dots + n_{\ell})^\gamma   \\
&= (n_1 + n_2 + \dots + n_{\ell})^\gamma \left( \Big( 1 + \frac{n_{{\ell} +1}}{n_1 + n_2 + \dots + n_{\ell}} \Big)^\gamma - 1 \right)  \\
&\approx \frac{\gamma  n_{{\ell} +1}}{(n_1 + n_2 + \dots + n_{\ell})^{1-\gamma}}  \to 0.
\end{align*}
Therefore, for any $\varepsilon >0$, there exists $k_0\geq 1$ such that for all $j>k_0$,
\begin{align}\label{ineq:nj}
n_{j} < \varepsilon (n_1 + n_2 + \dots + n_{j-1})^{1-\gamma}.
\end{align}
Then for any $k_0<j\leq {\ell}$
$$n_{j} < \varepsilon (n_1 + n_2 + \dots + n_{{\ell}})^{1-\gamma}.$$
This implies
\begin{align*}
S_{n_1 + n_2 +\dots + n_{\ell}}(x) &= 2^{n_1} + 2^{n_2} + \dots + 2^{n_{\ell}} - {\ell} \\
&\leq  M+{\ell} 2^{\epsilon (n_1 + n_2 + \dots + n_{{\ell}})^{1-\gamma}}-{\ell},
\end{align*}
with $M:=2^{n_1}+\cdots +2^{n_{k_0}}$.
Thus we have
\begin{equation}\label{form:S}
\frac{S_{n_1 + n_2 +\dots + n_{\ell}}(x)}{\Psi(n_1 + n_2 +\dots + n_{\ell})}
< \frac{M+{\ell} 2^{\epsilon (n_1 + n_2 + \dots + n_{{\ell}})^{1-\gamma}}-{\ell}}{2^{(n_1 + n_2 + \dots + n_{\ell})^\gamma}}.
\end{equation}
By observing $n_j\geq 1$, we deduce that the upper bound of (\ref{form:S}) converges to 0 for $1/2 \leq \gamma < 1$, a contradiction to \eqref{3.1}.
Hence $ E_\Psi({\beta})$ is an empty set.

\medskip


(II$_2$). $\beta=\infty$.
 Fix $\delta \in (\gamma, 1)$ and take a large integer $K$ such that $2^{K\delta}>1$. Consider the set of points such that at every position $2^k, k>K$ in their binary expansions, they have a string of zeros of length $2^{\delta k}$, i.e. $$
 E:=\Big\{x\in (0,1]: \epsilon_{2^k+1}=\cdots=\epsilon_{2^k+\lfloor 2^{k\delta}\rfloor}=0,\ {\text{for all}}\ k\ge K\Big\}.
 $$
  On the one hand, $E\subset E_{\psi}(\infty)$, since for any $n\in (2^k,2^{k+1}]$ for some $k\ge K$, \[
S_n(x)> 2^{2^{k\delta}}\ge 2^{(n/2)^\delta} \gg 2^{n^\gamma}.
\]
On the other hand, the set $E$ has dimension 1 guaranteed by Lemma \ref{l1}.
\medskip

(III). Suppose that $\Psi(n)=2^{n^{\gamma}}$ with $\gamma \ge 1$ and let $\beta\in (0,+\infty]$. Assume that there exists $x\in E_{\Psi}(\beta)$ for some $\beta\in (0,+\infty)$. Write the binary expansion of $x$ as (\ref{3.3}). 
Then by  \eqref{f1}, 
\begin{equation}\label{3.2}
\begin{split}
\frac{S_{n_1 + n_2 +\dots + n_{\ell}}(x)}{\Psi(n_1 + n_2 +\dots + n_{\ell})} &= \frac{2^{n_1} + 2^{n_2} + \dots + 2^{n_{\ell}}-\ell}{2^{(n_1 + n_2 + \dots + n_{\ell})^\gamma}} \to \beta, \\
\frac{S_{n_1 + n_2 +\dots + n_{\ell}- 1}(x)}{\Psi(n_1 + n_2 +\dots + n_{\ell}-1)} &= \frac{2^{n_1} + 2^{n_2} + \dots + 2^{n_{\ell}} -{\ell} -1 }{2^{(n_1 + n_2 + \dots + n_{\ell} -1)^\gamma}} \to \beta.
\end{split}\end{equation}
However,
$$
 \frac{2^{n_1} + 2^{n_2} + \dots + 2^{n_{\ell}} - {\ell}}{2^{n_1} + 2^{n_2} + \dots + 2^{n_{\ell}} - {\ell} -1 } \to 1 \ \text{ but } \ \frac{2^{(n_1 + n_2 + \dots + n_{\ell})^\gamma}}{2^{(n_1 + n_2 + \dots + n_{\ell} -1)^\gamma}} \ge 2,$$
which is a contradiction.  Hence $E_\Psi({\beta})$ is empty when $\beta\in (0,+\infty)$.

When $\beta=+\infty$, by (\ref{f1}), we have
\begin{equation}\label{3.15}
\liminf_{n\to\infty}\frac{S_n(x)}{2^n}\le 1.
\end{equation}
So,
\begin{equation*}
 \liminf_{n\to\infty}\frac{S_n(x)}{\Psi(n)}\le 1.
\end{equation*}
 This shows that $E_{\Psi}(\infty)$ is also empty.
\end{proof}

\bigskip
\section{The potential $1/x$}\label{sec:1overx}

In fact, the techniques in Section \ref{sec:fast} can be applied to the continuous potential $g: x\mapsto 1/x$ on $(0,1]$ which has a singularity at $0$.


\begin{proof}[Proof of Theorem \ref{t3}]
We first show that if $\Psi(n)$ is
one of the following
\[
\Psi(n)=n\log n, \ \Psi(n)=n^a \ (a>1), \ \Psi(n)=2^{n^{\gamma}} \ (0<\gamma<1/2),
\]
then for any $\beta  \in [0,\infty]$, $\dim_HF_\Psi({\beta})=1$.

We note that if $x\in (0,1]$ has binary expansion $x=[0^n1^s\dots]$, then $\varphi(x)=2^n$ and
\begin{align}\label{val:g}
2^n\leq g(x) \leq 2^n+2^{n-s+1}=2^n(1+2^{-s+1}).
\end{align}

In Lemma \ref{lem:2}, for an integer $W$, and for any integer $n\leq \log W$, we can construct instead of the words  $
w=(10^{t_1-1}1, 1 0^{t_2-1}1, \cdots, 10^{t_p-1}1)
$, the following word
$$
w=(10^{t_1-1}1^{s+1}, 1 0^{t_2-1}1^{s+1}, \cdots, 10^{t_p-1}1^{s+1}).
$$
Then the length of the word satisfies
\begin{align}\label{length:word}
|w|=\sum_{i=1}^p(t_i+s+1)\le p(t_1+s+1)\le (n+2) (\log W+s+2).
\end{align}
By (\ref{val:g}), for any $x\in I_{|w|}(w)$,
$$
W+s (n+2)\leq \sum_{j=0}^{|w|-1}g(T^jx) \leq  W(1+2^{-n})\cdot (1+2^{-s})+2s(n+2).
$$

For each $k\ge 1$,  we still write $$W_k:=\Psi(N_k)-\Psi(N_{k-1})$$
and let $n_k$, $s_k$ be a sequence of integers tending to $\infty$ such that
\begin{align}\label{cond-nksk}
n_k \cdot \frac{\log \big(\Psi(N_{k})-\Psi(N_{k-1})\big)}{N_{k}-N_{k-1}}\to 0,
\end{align}
\begin{align}\label{cond-nksk-2}
{n_k \cdot s_k \over N_{k}-N_{k-1}}\to 0,
\end{align}
and
\begin{align}\label{cond-nksk-3}
{n_k \cdot s_k \over \Psi(N_{k})-\Psi(N_{k-1})}\to 0.
\end{align}
By (\ref{f6}) and (\ref{f6'}), these two sequences of $n_k\geq 0$, $s_k\geq 0$ do exist.

Now for $W_k$ and $n_k$, $s_k$, let $w_k$ be the word given as above. Then 
by (\ref{cond-nksk}) and (\ref{cond-nksk-2}), the length $a_k$ of $w_k$ satisfies
\begin{equation}\label{f5-bis}
\begin{split}
a_k\le& (n_k+2)(\log W_k+s_k+2)\\
=&(n_k+2)\left(\log(\Psi(N_k)-\Psi(N_{k-1}))+s_k+2\right) \\
=&o(N_{k}-N_{k-1})
\end{split}
\end{equation}
and for any $x\in I_{a_k}(w_k)$,
\begin{equation}\label{f4-bis}
\begin{split}
W_k+s_k (n_k+2) &\leq \sum_{j=0}^{a_k-1}g(T^jx) \\ &\leq W_k (1+2^{-n_k})\cdot (1+2^{-s_k})+2s_k (n_k+2).
\end{split}\end{equation}
Hence by (\ref{cond-nksk-3}) we still have the same estimation:
\begin{align*}
\sum_{n=0}^{N_{k}-1}g(T^nx)=\Psi(N_{k})+o(\Psi(N_{k}))
\end{align*}
and the rest of the proof is the same as (I) of the proof of Theorem \ref{t2}.

\medskip
We can repeat the same arguments in Section \ref{sec:fast} and show that for potential $g$, the set $F_\Psi(\beta)$ is empty  if
$\beta\in (0, \infty), \Psi(n)=2^{n^\gamma} ( 1/2 \leq \gamma <1)$ or $\beta\in (0, \infty], \Psi(n)=2^{n^\gamma} (\gamma \geq 1)$.

In fact,  by definition, if there exists $x \in F_\Psi({\beta})$, with its binary expansion
$$x=[0^{n_1-1} 10^{n_2-1} 1 \cdots 0^{n_{\ell}-1} 1 0^{n_{\ell+1}-1} 1\cdots],$$
then
\begin{equation*}
\begin{split}
\frac{S_{n_1 + n_2 +\dots + n_{\ell}}g(x)}{\Psi(n_1 + n_2 +\dots + n_{\ell})}  \to \beta, \quad
\frac{S_{n_1 + n_2 +\dots + n_{\ell}+1}g(x)}{\Psi(n_1 + n_2 +\dots + n_{\ell}+1)} \to \beta.
\end{split}\end{equation*}
Thus
\begin{equation*}
\begin{split}
\frac{S_{n_1 + n_2 +\dots + n_{\ell}}g(x)}{S_{n_1 + n_2 +\dots + n_{\ell}+1}g(x)} \to 1.
\end{split}\end{equation*}
Observing $\varphi\leq g \leq 2\varphi$, we have
\begin{equation*}
\begin{split}
\frac{2^{n_{{\ell}+1}}}{S_{n_1 + n_2 +\dots + n_{\ell}}g(x)} \to 0,
\end{split}\end{equation*}
which then implies
\begin{equation*}
\frac{S_{n_1 + n_2 +\dots + n_{\ell}}g(x)}{S_{n_1 + n_2 +\dots + n_{\ell+1}}g(x)} \to 1.
\end{equation*}
By the definition of $x\in F_\Psi(\beta)$, we have
\begin{equation*}
\begin{split}
\frac{\Psi({n_1 + n_2 +\dots + n_{{\ell}+1}})}{\Psi({n_1 + n_2 +\dots + n_{\ell}})} \to 1.
\end{split}\end{equation*}
This further implies the same inequality with \eqref{ineq:nj} and the rest of proof is the same as (II$_1$) and (III) by noting $\varphi\leq g \leq 2\varphi$.

\medskip
For $\Psi(n)=2^{n^\gamma} ( 1/2 \leq \gamma <1)$, we can also prove $\dim_HF_\Psi(\infty)=1$ by the same proof as (II$_2$).
\end{proof}

%
%

\end{document}